\documentclass[oneside,english]{amsart}
\usepackage{lmodern}

\usepackage[T1]{fontenc}
\usepackage[latin9]{inputenc}
\usepackage{enumitem}
\usepackage{amsthm}
\usepackage{amssymb}
\usepackage{graphicx}
\PassOptionsToPackage{normalem}{ulem}
\usepackage{ulem}

\makeatletter
\numberwithin{equation}{section}
\numberwithin{figure}{section}
  \theoremstyle{remark}
  \newtheorem*{acknowledgement*}{\protect\acknowledgementname}
\theoremstyle{plain}
\newtheorem{thm}{\protect\theoremname}[section]
  \theoremstyle{definition}
  \newtheorem{defn}[thm]{\protect\definitionname}
  \theoremstyle{plain}
  \newtheorem{assumption}[thm]{\protect\assumptionname}
  \theoremstyle{plain}
  \newtheorem{lem}[thm]{\protect\lemmaname}
  \theoremstyle{plain}
  \newtheorem{fact}[thm]{\protect\factname}
  \theoremstyle{remark}
  \newtheorem{rem}[thm]{\protect\remarkname}
  \theoremstyle{plain}
  \newtheorem{cor}[thm]{\protect\corollaryname}
  \theoremstyle{plain}
  \newtheorem{prop}[thm]{\protect\propositionname}
  \theoremstyle{definition}
  \newtheorem{example}[thm]{\protect\examplename}
 \newlist{casenv}{enumerate}{4}
 \setlist[casenv]{leftmargin=*,align=left,widest={iiii}}
 \setlist[casenv,1]{label={{\itshape\ \casename} \arabic*.},ref=\arabic*}
 \setlist[casenv,2]{label={{\itshape\ \casename} \roman*.},ref=\roman*}
 \setlist[casenv,3]{label={{\itshape\ \casename\ \alph*.}},ref=\alph*}
 \setlist[casenv,4]{label={{\itshape\ \casename} \arabic*.},ref=\arabic*}
  \theoremstyle{plain}
  \newtheorem{question}[thm]{\protect\questionname}

\usepackage{amsmath}

\usepackage{a4wide}
\usepackage{url}
\makeatother

\usepackage{babel}
  \providecommand{\acknowledgementname}{Acknowledgement}
  \providecommand{\assumptionname}{Assumption}
  \providecommand{\corollaryname}{Corollary}
  \providecommand{\definitionname}{Definition}
  \providecommand{\examplename}{Example}
  \providecommand{\factname}{Fact}
  \providecommand{\lemmaname}{Lemma}
  \providecommand{\propositionname}{Proposition}
  \providecommand{\questionname}{Question}
  \providecommand{\remarkname}{Remark}
 \providecommand{\casename}{Case}
\providecommand{\theoremname}{Theorem}

\begin{document}
\global\long\def\Aut{\operatorname{Aut}}
\global\long\def\calB{\mathcal{B}}
\global\long\def\calI{\mathcal{I}}
\global\long\def\ball#1#2{\calB(#1, #2)}
\global\long\def\calU{\mathcal{U}}
\global\long\def\calV{\mathcal{V}}
\global\long\def\Cantorspace{\functions{\omega}{2}}
\global\long\def\closure#1{\overline{#1}}
\global\long\def\concatenation{\mathrel{\smallfrown}}
\global\long\def\constantsequence#1#2{\left(#1\right){}^{#2}}
\global\long\def\diameter#1{\mathrm{diam}\left(#1\right)}
\global\long\def\Ezero{\mathbb{E}_{0}}
\global\long\def\from{\colon}
\global\long\def\Fsigma{F_{\sigma}}
\global\long\def\functions#1#2{#2^{#1}}
\global\long\def\heightcorrection#1{\raisebox{0pt}[0pt][0pt]{#1}}
\global\long\def\ideal#1{\calI_{#1}}
\global\long\def\image#1#2{#1\left[#2\right]}
\global\long\def\inverse#1{#1^{-1}}
\global\long\def\mathand{\text{ and }}
\global\long\def\N{\mathbb{N}}
\global\long\def\orbitequivalencerelation#1#2{E_{#1}^{#2}}
\global\long\def\pair#1#2{\left(#1,#2\right)}
\global\long\def\relationpower#1#2{#1^{\left(#2\right)}}
\global\long\def\restriction#1#2{#1 \upharpoonright#2}
\global\long\def\setcomplement#1{\twiddle#1}
\global\long\def\sets#1#2{\left[#2\right]{}^{#1}}
\global\long\def\suchthat{\mid}
\global\long\def\twiddle{\raisebox{1pt}{\scalebox{.75}{$\mathord{\sim}$}}}
\global\long\def\C{\mathfrak{C}}
\global\long\def\acl{\operatorname{acl}}
\global\long\def\tp{\operatorname{tp}}
\global\long\def\qf{\operatorname{qf}}
\global\long\def\Nn{\mathbb{N}}
\global\long\def\id{\operatorname{id}}
\global\long\def\SS{\mathcal{P}}
\global\long\def\EM{\operatorname{EM}}
\global\long\def\dcl{\operatorname{dcl}}
\global\long\def\Autf{\operatorname{Aut}f_{L}}
\global\long\def\eq{\operatorname{eq}}
\global\long\def\image{\mbox{image}}

\global\long\def\pamod#1{\pmod#1}

\global\long\def\nf{\mbox{nf}}
\global\long\def\Uu{\mathcal{U}}
\global\long\def\dom{\operatorname{dom}}
\global\long\def\concat{\smallfrown}
\global\long\def\Nn{\mathbb{N}}
\global\long\def\mathrela#1{\mathrel{#1}}
\global\long\def\twiddle{\mathord{\sim}}
\global\long\def\stab{\operatorname{stab}}
 \global\long\def\x{\times}
\global\long\def\diam{\operatorname{diam}}
\global\long\def\EZero{\mathbb{E}_{0}}
\global\long\def\sequence#1#2{\left\langle #1\left|\,#2\right.\right\rangle }
\global\long\def\set#1#2{\left\{  #1\left|\,#2\right.\right\}  }
\global\long\def\cardinal#1{\left|#1\right|}
\global\long\def\calO{\mathcal{O}}
\global\long\def\mathordi#1{\mathord{#1}}
\global\long\def\Ezero{\EZero}
\global\long\def\xx{\mathbf{x}}

\global\long\def\NTPT{\operatorname{NTP}_{\operatorname{2}}}
\global\long\def\ist{\operatorname{ist}}
\global\long\def\C{\mathfrak{C}}
\global\long\def\alt{\operatorname{alt}}
\global\long\def\Ff{\mathbb{F}}
\global\long\def\Ll{\mathfrak{L}}
\global\long\def\calU{\mathcal{U}}
\global\long\def\Qq{\mathbb{Q}}
\global\long\def\AGLn{AGL_{n}\left(\Qq\right)}
\global\long\def\AGLt{AGL_{2}\left(\Qq\right)}
\global\long\def\AGL#1{AGL_{#1}\left(\Qq\right)}
\global\long\def\Zz{\mathbb{Z}}
\global\long\def\GL#1{GL_{#1}\left(\Qq\right)}
\global\long\def\bb{\mathfrak{b}}
\global\long\def\QqOmega{\Qq^{<\omega}}
\global\long\def\PGL#1{PGL_{#1}\left(\Qq\right)}
\global\long\def\PGaL#1#2{P\Gamma L_{#1}\left(#2\right)}
\global\long\def\AGalL#1#2{A\Gamma L_{#1}\left(#2\right)}

\title{The affine and projective groups are maximal }

\author{Itay Kaplan and Pierre Simon}
\begin{abstract}
We show that the groups $\AGL n$ (for $n\geq2$) and $\PGL n$ (for
$n\geq3$), seen as closed subgroups of $S_{\omega}$, are maximal-closed.
\end{abstract}

\thanks{The first author would like to thank the Israel Science foundation
for partial support of this research (Grant no. 1533/14).\\
The research leading to these results has received funding from the
European Research Council under the European Union\textquoteright s
Seventh Framework Programme (FP7/2007- 2013)/ERC Grant Agreement No.
291111.}


\address{Itay Kaplan \\
The Hebrew University of Jerusalem\\
Einstein Institute of Mathematics \\
Edmond J. Safra Campus, Givat Ram\\
Jerusalem 91904, Israel}

\email{kaplan@math.huji.ac.il}

\urladdr{https://sites.google.com/site/itay80/ }

\address{Pierre Simon\\
Institut Camille Jordan\\
Université Claude Bernard - Lyon 1\\
43 boulevard du 11 novembre 1918\\
69622 Villeurbanne Cedex, France}

\email{simon@math.univ-lyon1.fr}

\urladdr{http://www.normalesup.org/\textasciitilde{}simon/}

\maketitle
Subgroups of the infinite symmetric group $S_{\omega}$ (or more generally,
$S_{\kappa}$ for any cardinal) have been studied extensively, and
a lot of work has been done concerning maximal subgroups. For instance,
in \cite{MacphersonPragerCheryl}, the authors proved that if a subgroup
$G\leq S_{\omega}$ is not highly transitive, then it is contained
in a maximal subgroup. In particular, closed (in the usual product
topology) proper subgroups of $S_{\omega}$ are contained in maximal
subgroups. Macpherson and Neumann \cite[Observation 6.1]{NeumannMacpherson}
proved that if $H\leq S_{\omega}$ is maximal and non-open, $H$ must
be highly transitive. In particular, closed non-open subgroups cannot
be maximal. 

Later, in \cite{BaumgartnerShelahSimon}, the authors proved that
it is consistent with ZFC that for every $\kappa\geq\omega$, there
is a subgroup $G\leq S_{\kappa}$ which is not contained in a maximal
subgroup. As far as we know, it is an open question whether this can
be proved in ZFC. However, not much is known about maximal-closed
subgroups (i.e., groups that are maximal in the family of closed subgroups
of $S_{\omega}$). In \cite[Example 7.10]{NeumannMacpherson} several
examples of such groups are given, including $AGL_{\omega}\left(p\right)$
for a prime $p$ (meaning the affine group acting on the infinite
dimensional vector space over $\Ff_{p}$) and $\PGaL{\omega}q$ for
a prime power $q$. 

If $G\leq S_{\omega}$ is the automorphism group of an $\omega$-categorical
structure $M$, then closed supergroups of $M$ in $S_{\omega}$ are
in one-to-one correspondence with reducts of $M$. The full classification
of reducts is known for a number of $\omega$-categorical structures
(\cite{Cameron,Thomas_graph,Thomas_hyper,reductsRamsey}). Nevertheless,
the main question asked by Thomas over 20 years ago remains unresolved:
we do not know if every homogeneous structure on a finite relational
language has only finitely many reducts.

In another direction, Junker and Ziegler \cite{DBLP:journals/jsyml/JunkerZ08}
asked for a converse to that question: if a structure is not $\omega$-categorical,
does it necessarily have infinitely many reducts. Let $M$ be a countable
structure in a language $L$ and let $\Aut\left(M\right)\leq S_{\omega}$
be its automorphism group. A \emph{(proper) group reduct} of $M$
is a structure $M'$ with the same domain as $M$ whose automorphism
group $\Aut\left(M'\right)$ properly contains $\Aut\left(M\right)$.
It is non-trivial if it is not the whole of $S_{\omega}$. On the
other hand, a \emph{(proper) definable reduct} of $M$ is a proper
reduct of the structure $M$ in the usual model theoretic sense (see
Definition \ref{def:reduct}). It is non-trivial if it is not interdefinable
with equality. In general, classifying group reducts and definable
reducts of a given structure are two independent questions: a group
reduct need not be definable, and a definable reduct need not admit
additional automorphisms. Thus Junker and Ziegler's question breaks
into two sub-questions.

In a recent paper \cite{Bodirsky2013}, Bodirsky and Macpherson have
answered the two versions of Junker and Ziegler's question negatively:
there are non-$\omega$-categorical structures which admit no reducts
in either sense. The automorphism groups involved in their work are
of a different flavor than those that appear in the study of $\omega$-categorical
structures. Whereas the techniques used in the study of reducts of
$\omega$-categorical structures involve mainly Ramsey properties,
Bodirsky and Macpherson make use of the theory of Jordan groups. Jordan
groups can be classified according to the type of structure that they
preserve. Bodirsky and Macpherson focus on automorphism groups of
$D$-relations.

In this paper, we study some members of another family of Jordan groups:
that of automorphism groups of Steiner systems. The most natural Steiner
systems are the family of lines in an affine or projective space,
giving rise to the groups $\AGLn$ and $\PGL n$. We show, making
an essential use of Adeleke and Macpherson's classification theorem
of Jordan groups, that those groups are maximal-closed in $S_{\omega}$
(in dimensions larger than 1). This answers a question of Macpherson
and Bodirsky from \cite{Bodirsky2013} about examples of \emph{countable
}maximal-closed subgroups of $S_{\omega}$. We also deduce from these
results that the structures $\left(\mathbb{Q}^{n},f\right)$ for $n\geq1$,
where $f\left(x,y,z\right)=x+y-z$, admit no definable reduct, which
answers another question from \cite{Bodirsky2013}.

After our paper was submitted it came to our attention that in a recent
work, Bogomolov and Rovinsky \cite{bogomolovRovinsky} proved that
for $n\geq3$, $\PGaL nF$ is a maximal-closed subgroup of the group
of all permutations of $F$ for any infinite field $F$, using completely
different methods. In particular their result implies ours on $\PGL n$.
However, our result about $\AGL n$ does not extend to all infinite
fields (see an example after Definition \ref{def:AGL}).

A few words about the proof. The classification theorem for Jordan
groups (Theorem \ref{fac:Macpharson-and-Adeleke}) says that any 3-transitive
Jordan group must preserve a structure of one of the following types:
a cyclic separation relation, a $D$-relation, a $k$-Steiner system,
or a limit of Steiner systems (all of which will be defined later).
Thus to prove that $\AGL n$ say is maximal-closed, we will show that
any group properly containing it does not preserve any one of these
structures. The separation relation, $D$-relation and Steiner system
are treated with combinatorial \emph{ad-hoc} constructions. To rule
out limits of Steiner systems, we use a slightly more general (or
generalizable) construction which shows that if a (sufficiently transitive)
group containing $\AGL n$ preserves a limit of Jordan sets (see Definition
\ref{def:weak limit}), then it actually preserves a Steiner system.
We believe that this argument could apply to other groups as well.
\begin{acknowledgement*}
We would like to thank David Evans for noticing Corollary \ref{cor:David Evans},
which states that $\left(\Qq,f\right)$ has no definable reducts. 
\end{acknowledgement*}

\section{\label{sec:AGL}$\protect\AGL n$ is maximal}
\begin{defn}
\label{def:AGL}For a vector space $V$ over a field $F$, $AGL\left(V\right)$
is the group of permutations of $V$ consisting of maps of the form
$x\mapsto Tx+b$ where $T\in GL\left(V\right)$ (i.e., an invertible
linear map) and $b\in V$. For $n<\omega$, let $\AGL n$ be $AGL\left(\Qq^{n}\right)$
and also let $\AGL{\omega}$ be $AGL\left(\QqOmega\right)$ (where
$\QqOmega$ is the vector space of infinite sequences with finite
support). Similarly we define $\GL n$. 
\end{defn}
Note that for a vector space $V$ over $\Qq$, $AGL\left(V\right)$
is the group of automorphisms of the structure $\left(V,f\right)$
where $f\left(x,y,z\right)=x+y-z$.

The group $S_{\omega}$ of all permutations of $\mathbb{N}$ is a
topological group where the topology is the product topology (i.e.,
the one induced by $\Nn^{\Nn}$). In fact for a natural metric it
is a Polish group --- $d\left(\sigma_{1},\sigma_{2}\right)=1/\left(\min\set n{\sigma_{1}\left(n\right)\neq\sigma_{2}\left(n\right)\vee\sigma_{1}^{-1}\left(n\right)\neq\sigma_{2}^{-1}\left(n\right)}+1\right)$.
For $n\leq\omega$, we can embed $\AGL n$ naturally (after choosing
a bijection from $\Qq^{n}$, or $\QqOmega$, to $\Nn$) as a subgroup
of $S_{\omega}$. It is then a closed subgroup, since it is the automorphism
group of a structure. Note also that for $n<\omega$ it is countable,
while for $n=\omega$ it is uncountable.

\paragraph*{Examples}
\begin{itemize}
\item \cite[Proposition 5.7]{Bodirsky2013}Take the structure $\left(\Qq,f\right)$,
where $f\left(x,y,z\right)=x+y-z$. Then its automorphism group is
$\AGL 1$. It is also countable and closed, but it is not maximal.
For a prime $p$, let $C_{p}\left(x,y,z\right)$ be $v_{p}\left(y-x\right)<v_{p}\left(z-y\right)$,
where $v_{p}$ is the $p$-adic valuation. This is a $C$-relation
on $\Qq$, and its automorphism group is uncountable:

For every $\eta:\omega\to2$, let $\sigma_{\eta}:\Qq\to\Qq$ be the
following map: for each $x\in\Qq$, let $\sum_{i=-n}^{\infty}a_{i}p^{i}$
be its (unique) $p$-adic expansion (where $n\in\Nn$, $a_{i}<p$
and $a_{-n}\neq0$), and let $\sigma_{\eta}\left(x\right)=a_{-n}p^{-n}+\sum_{i=-\left(n+1\right)}^{0}a_{i}^{\eta\left(-i\right)}p^{i}+\sum_{i=1}^{\infty}a_{i}p^{i}$
where $a_{i}^{0}=a_{i}$ and $a_{i}^{1}=a_{i}+1\pmod p$. Then $\set{\sigma_{\eta}}{\eta:\omega\to2}$
is a set of distinct automorphisms of $\left(\Qq,C\right)$. 

\item Now take the group $AGL_{2}\left(R\right)$ where $R$ is any field
extension of $\Qq$. The plane $R^{2}$ can be seen as a $\Qq$-vector
space, and so the group of affine transformation over $\Qq$ properly
contains $AGL_{2}\left(R\right)$ hence that group is not maximal. 
\end{itemize}
In this section we will prove:
\begin{thm}
\label{thm:main}For $2\leq n\leq\omega$, $\AGLn$ is a maximal-closed
subgroup of $S_{\omega}$. So for $2\leq n<\omega$, $\AGL n$ is
a countable maximal-closed group. \end{thm}
\begin{assumption}
Throughout this section, we fix $2\leq n\leq\omega$. Let $\Omega=\Qq^{n}$
in case $n<\omega$ and $\QqOmega$ in case $n=\omega$. 
\end{assumption}
As we noted, $\AGL n$ is closed. In order to prove that it is maximal,
it is enough to show that any group $G$ containing it is \emph{highly
transitive}, i.e., $k$-transitive for every $k<\omega$. So this
is what we will do. 
\begin{defn}
For 3 collinear points $\left(a,b,c\right)\in\Omega^{3}$ we will
say that they have \emph{ratio $r\in\Qq$} if $c=ra+\left(1-r\right)b$. 
\end{defn}
The following lemma (whose proof we leave to the reader) describes
the orbits of triples in $\Omega$ under the action of $\AGL n$.
\begin{lem}
\label{lem:orbits}$\AGL n$ is 2-transitive. Moreover, the orbits
of its action on triples of distinct elements from $\Omega$ are:
\begin{itemize}
\item For each $r\in\Qq\backslash\left\{ 0,1\right\} $ and $a\neq b$,
$\set{\left(a,b,c\right)}{c=ra+\left(1-r\right)b}$ --- all triples
of ratio $r$.
\item $\set{\left(a,b,c\right)}{\left(a,b,c\right)\mbox{ are not collinear}}$. 
\end{itemize}
\end{lem}
\begin{fact}
\label{fac:The-fundamental-theorem} \cite{MR0137019,MR1629468}(The
fundamental theorem of affine geometry) For a permutation $\sigma$
of $\Omega$, $\sigma\in\AGL n$ iff $\sigma$ preserves lines iff
$\sigma$ takes triples of collinear points to collinear points. \end{fact}
\begin{rem}
In the proof of Fact \ref{fac:The-fundamental-theorem}, one first
proves that if $\sigma$ maps lines to lines (equivalently 3 collinear
points to 3 collinear points), then it must take affine planes to
affine planes. This is done by noting that if two lines intersect,
then they lie on the same plane. Then, the main point is to prove
additivity of $\sigma$, which is done geometrically, i.e., assuming
that $\sigma\left(0\right)=0$, for any $u,v\neq0$ in $\Omega$,
$u+v$ is the intersection of the lines $L_{1}$ --- the unique line
in the plane containing $u,v$ which is parallel to $v$ containing
$u$ --- and $L_{2}$ which is defined similarly. Since $\sigma\left(u+v\right)$
is in the same plane spanned by $\sigma(0)$, $\sigma\left(u\right)$
and $\sigma\left(v\right)$, the same geometrical property holds for
$\sigma\left(u+v\right)$. \end{rem}
\begin{lem}
\label{lem:if sigma doesn't preserve lines}Let $\sigma$ be any permutation
of $\Omega$. Suppose that $L$ is a line containing $0$ which $\sigma$
does not map to a line. Then for any $r\in\Qq\backslash\left\{ 0,1\right\} $
there are 3 collinear point of ratio $r$ on $L$ which are sent by
$\sigma$ to 3 non-collinear points. \end{lem}
\begin{proof}
For $a\neq b\in\Omega$, denote by $L\left(a,b\right)=\set{c\in\Omega}{\exists r\left(c=ra+\left(1-r\right)b\right)}$
the unique line that contains both $a$ and $b$. Let $E$ be the
equivalence relation defined on $L\backslash\left\{ 0\right\} $ by:
$a$ and $b$ are equivalent if $L\left(\sigma\left(0\right),\sigma\left(a\right)\right)=L\left(\sigma\left(0\right),\sigma\left(b\right)\right)$.
Equivalently, $a$ and $b$ are equivalent if $\sigma\left(b\right)\in L\left(\sigma\left(0\right),\sigma\left(a\right)\right)$. 

Suppose for contradiction that for some $r\in\Qq\backslash\left\{ 0,1\right\} $,
all 3 collinear points of ratio $r$ are sent by $\sigma$ to collinear
points. This means that for any $a\in L\backslash\left\{ 0\right\} $,
$\left[a\right]_{E}$ contains $ra$ (as $\left(a,0,ra\right)$ has
ratio $r$), $\left(1-r\right)a$ (as $\left(0,a,\left(1-r\right)a\right)$
has ratio $r$), $\frac{1}{r}a$ (as $\left(\frac{1}{r}a,0,a\right)$
has ratio $r$) and $\frac{1}{1-r}a$ (as $\left(0,\frac{1}{1-r}a,a\right)$
has ratio $r$).

Next, we note that $a\mathrela E-a$ for $a\in L\backslash\left\{ 0\right\} $.
Since $0=a-a=r\left(\frac{1}{r}a\right)+\left(1-r\right)\left(-\frac{1}{1-r}a\right)$,
$\left(\frac{1}{r}a,-\frac{1}{1-r}a,0\right)$ has ratio $r$, so
$\sigma\left(0\right)\in L\left(\sigma\left(\frac{1}{r}a\right),\sigma\left(-\frac{1}{1-r}a\right)\right)$
(note that $\frac{1}{r}\neq-\frac{1}{1-r}$). This implies that $\frac{1}{r}a\mathrela E-\frac{1}{1-r}a$.
But we already know that $\frac{1}{r}a\mathrela Ea$ and $\frac{1}{1-r}\left(-a\right)\mathrela E-a$.
Together we are done.

Now we will show that if $b\mathrela Ea$ then $a+b\mathrela Ea$
for $a,b\in L\backslash\left\{ 0\right\} $. There are two cases to
consider. First, assume that $\frac{1}{r}a=\frac{1}{1-r}b$: this
implies that $a+b=\frac{1}{r}a$, which we already know is $E$-equivalent
to $a$. If not, then as $\left(\frac{1}{r}a,\frac{1}{1-r}b,a+b\right)$
is a tuple of 3 collinear points of ratio $r$, our assumption implies
that $\sigma\left(a+b\right)$ is in $L\left(\sigma\left(\frac{1}{r}a\right),\sigma\left(\frac{1}{1-r}b\right)\right)$.
Since $\frac{1}{1-r}b\mathrela Eb\mathrela Ea\mathrela E\frac{1}{r}a$,
this line equals $L\left(\sigma\left(0\right),\sigma\left(a\right)\right)$
and we are done.

We conclude that for all $a\in L\backslash\left\{ 0\right\} $ and
$n\in\Zz$, $na\mathrela Ea$. But then $\frac{1}{n}a\mathrela Ea$
as well, so $\left[a\right]_{E}=L\backslash\left\{ 0\right\} $. This
contradicts our assumption. \end{proof}
\begin{cor}
\label{cor:3-trans}If $G$ is a permutation group of $\Omega$ properly
containing $\AGL n$, then $G$ is 3-transitive.\end{cor}
\begin{proof}
By Fact \ref{fac:The-fundamental-theorem}, there is some $\sigma\in G$
which does not preserve lines. By precomposing with an element from
$\AGL n$, we may assume that there is a line $L$ containing $0$
which $\sigma$ does not preserve. The corollary follows from Lemma
\ref{lem:orbits} and Lemma \ref{lem:if sigma doesn't preserve lines}. 
\end{proof}
Now we are going to use a theorem of Adeleke and Macpherson about
the classification of Jordan groups. 
\begin{defn}
Suppose $G$ is a group acting on a set $X$. A set $\emptyset\neq A\subseteq X$
of size at least 2 is a\emph{ Jordan set }if the pointwise stabilizer
$G_{X\backslash A}$ acts transitively on $A$. If there is a Jordan
set $A\subseteq X$ such that for all $k\in\Nn$ for which $G$ is
$k$-transitive, $\left|X\backslash A\right|\geq k$, then $\left(G,X\right)$
(or just $G$) is called a \emph{Jordan group}. 
\end{defn}
Let $A$ be an affine subspace of $\Omega$ of dimension $<n$. Then
$\Omega\backslash A$ is a Jordan set. To see this, assume without
loss that $A$ is a linear subspace (i.e., $0\in A$). Then, let $B$
be a basis of $A$, and let $x,y\notin A$. Then $B\cup\left\{ x\right\} $
and $B\cup\left\{ y\right\} $ are both independent, so there is some
$\sigma\in\GL n$ taking $x$ to $y$ while fixing $B$ (so also $A$).
 Hence $\Omega\backslash A$ witnesses that $\AGL n$ is a Jordan
group. In fact, any group $G$ containing $\AGL n$ is a Jordan group
by the same argument. 

Let us now turn to the classification theorem of Adeleke and Macpherson:
\begin{fact}
\label{fac:Macpharson-and-Adeleke}\cite[Theorem 1.0.2]{MR1357089}(Adeleke
and Macpherson, 1995) Suppose $G$ is an infinite 3-transitive Jordan
group acting on a space $X$ which is not highly transitive. Then
$G$ must preserve on $X$ one of the following structures:
\begin{enumerate}
\item a cyclic separation relation;
\item a $D$ relation;
\item a Steiner $k$ system for $k\geq2$;
\item a limit of Steiner systems. 
\end{enumerate}
\end{fact}
For the reader who looks at the reference in \cite{MR1357089}, the
relevant clauses in Theorem 1.0.2 there are (v) and (vi). Also note
that $G$ is automatically primitive (i.e., does not preserve any
non-trivial equivalence relation on $\Omega$), since it is 2-transitive. 

We will show that in fact $\AGL n$ does not preserve any of the structures
in (1)--(3) (except the case $k=2$ in (3)), and that any group properly
containing it does not preserve a structure of the form (3) or (4).
Using Corollary \ref{cor:3-trans} (3-transitivity) we conclude that
any such group is highly transitive. 

So in each following subsection we will rule out one of these structures.
Our definitions are all taken from \cite{MR1632579} except for that
of a limit of Steiner systems, which is taken from \cite{MR1357089}.

\subsection{\label{sub:A-cyclic-separation}A cyclic separation relation}
\begin{defn}
A quaternary relation $S$ defined on a set $X$ is a \emph{cyclic
separation relation} if it satisfies the following conditions for
all $\alpha,\beta,\gamma,\delta,\varepsilon\in X$:\end{defn}
\begin{itemize}
\item [(s1)]$S\left(\alpha,\beta;\gamma,\delta\right)\Rightarrow S\left(\beta,\alpha;\gamma,\delta\right)\land S\left(\gamma,\delta;\alpha,\beta\right)$;
\item [(s2)]$S\left(\alpha,\beta;\gamma,\delta\right)\land S\left(\alpha,\gamma;\beta,\delta\right)\Leftrightarrow\beta=\gamma\vee\alpha=\delta$;
\item [(s3)]$S\left(\alpha,\beta;\gamma,\delta\right)\Rightarrow S\left(\alpha,\beta;\gamma,\varepsilon\right)\vee S\left(\alpha,\beta;\delta,\varepsilon\right)$;
\item [(s4)]$S\left(\alpha,\beta;\gamma,\delta\right)\vee S\left(\alpha,\gamma;\delta,\beta\right)\vee S\left(\alpha,\delta;\beta,\gamma\right)$.
\end{itemize}
The idea is that there is, in the background, a circle $C$, and $S\left(\alpha,\beta,\gamma,\delta\right)$
holds iff $\gamma,\delta$ are on different ``components'' of $C\backslash\left\{ \alpha,\beta\right\} $.
See Figure \ref{fig:A-cyclic-separation}. 

\begin{figure}[h]
\includegraphics{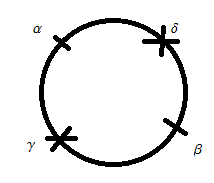}\protect\caption{\label{fig:A-cyclic-separation}A cyclic separation relation}
\end{figure}

\begin{lem}
Suppose that $S$ is a cyclic separation relation on a set $X$. Then
the following holds for all distinct $\alpha,\beta,\gamma,\delta,\delta'$
in $X$:\end{lem}
\begin{itemize}
\item [(s5)]$S\left(\alpha,\beta;\gamma,\delta\right)\land S\left(\alpha,\beta;\gamma,\delta'\right)\Rightarrow S\left(\alpha,\delta;\beta,\delta'\right)\vee S\left(\alpha,\delta';\beta,\delta\right)$.\end{itemize}
\begin{proof}
One easily sees that this property holds by observing Figure \ref{fig:A-cyclic-separation}.
However, we will give a formal proof. We first claim that a variant
of (s3) holds:
\begin{itemize}
\item [(s3')]For all distinct $\alpha,\beta,\gamma,\delta,\varepsilon\in X$,
if $S\left(\alpha,\beta;\gamma,\delta\right)$ holds then exactly
one of $S\left(\alpha,\beta;\gamma,\varepsilon\right)$ or $S\left(\alpha,\beta;\delta,\varepsilon\right)$
holds. 
\end{itemize}
Indeed, suppose for contradiction that both $S\left(\alpha,\beta;\gamma,\varepsilon\right)$
and $S\left(\alpha,\beta;\delta,\varepsilon\right)$ hold. By (s3)
and (s1) applied to $S\left(\alpha,\beta;\gamma,\varepsilon\right)$,
either $S\left(\delta,\beta;\gamma,\varepsilon\right)$ or $S\left(\alpha,\delta;\gamma,\varepsilon\right)$
holds. Suppose the former occurs. Then it cannot be that $S\left(\gamma,\beta;\varepsilon,\delta\right)$
(else $S\left(\beta,\delta;\gamma,\varepsilon\right)$ and $S\left(\beta,\gamma;\delta,\varepsilon\right)$
by (s1), contradicting (s2)). So applying (s3) (and (s1)) to $S\left(\alpha,\beta;\delta,\varepsilon\right)$
we get $S\left(\alpha,\gamma;\delta,\varepsilon\right)$. Since $S\left(\alpha,\gamma;\beta,\delta\right)$
is impossible by (s2), by applying (s3) to $S\left(\alpha,\gamma;\varepsilon,\delta\right)$
we get $S\left(\alpha,\gamma;\beta,\varepsilon\right)$, which contradicts
$S\left(\alpha,\beta;\gamma,\varepsilon\right)$ by (s2). The other
possibility is that $S\left(\alpha,\delta;\gamma,\varepsilon\right)$
holds, and it leads to a similar contradiction, by replacing $\alpha$
and $\beta$ in the argument. 

Now suppose $\alpha,\beta,\gamma,\delta,\delta'\in X$, and that $S\left(\alpha,\beta;\gamma,\delta\right)$
and $S\left(\alpha,\beta;\gamma,\delta'\right)$ hold. If the conclusion
does not hold, then by (s4), $S\left(\alpha,\beta;\delta,\delta'\right)$
holds. But then this contradicts (s3'), since now we can replace both
$\delta$ and $\delta'$ with $\gamma$. \end{proof}
\begin{prop}
\label{prop:No S relation} $\AGL n$ does not preserve a cyclic separation
relation on $\Omega$. \end{prop}
\begin{proof}
Suppose it does. 

Let $L$ be a line, and choose 3 points on it $a,b,c\in L$. Let $d$
be any point which is not on $L$. By (s4), we may assume that $S\left(a,b;c,d\right)$
holds. Since $\Omega\backslash L$ is a Jordan set, the same is true
for any $d'\notin L$. Finally, by (s5) we get that $S\left(a,d;b,d'\right)$
or $S\left(a,d';b,d\right)$ holds for any $d,d'\notin L$. By translation,
we may assume that $0\in L$, and that $b=-a$. By applying $GL_{n}\left(\Qq\right)$,
we may assume that $a=e_{1}$, where $\set{e_{i}}{i<n}$ form the
standard basis for $\Omega$. Now choose $d=e_{2}$, $d'=-e_{2}$. 

So we have $S\left(e_{1},e_{2};-e_{1},-e_{2}\right)$ or $S\left(e_{1},-e_{2};-e_{1},e_{2}\right)$.
There is some $\sigma\in\GL n$ that maps $e_{2}$ to $-e_{2}$ fixing
$e_{1}$. Then $\sigma\left(-e_{2}\right)=e_{2}$. So in any case
we contradict (s2) by applying $\sigma$. 
\end{proof}

\subsection{\label{sub:A-D-relation}A $D$-relation}
\begin{defn}
A quaternary relation $D$ defined on a set $X$ is a\emph{ $D$-relation}
if it satisfies the following conditions for all $\alpha,\beta,\gamma,\delta,\varepsilon\in X$:
\begin{itemize}
\item [(d1)]$D\left(\alpha,\beta;\gamma,\delta\right)\Rightarrow D\left(\beta,\alpha;\gamma,\delta\right)\wedge D\left(\gamma,\delta;\alpha,\beta\right)$;
\item [(d2)]$D\left(\alpha,\beta;\gamma,\delta\right)\Rightarrow\neg D\left(\alpha,\gamma;\beta,\delta\right)$;
\item [(d3)]$D\left(\alpha,\beta;\gamma,\delta\right)\Rightarrow D\left(\varepsilon,\beta;\gamma,\delta\right)\vee D\left(\alpha,\beta;\gamma,\varepsilon\right)$;
\item [(d4)]$\left(\alpha\neq\gamma\land\beta\neq\gamma\right)\Rightarrow D\left(\alpha,\beta;\gamma,\gamma\right)$;
\item [(d5)]$\alpha,\beta,\gamma$ distinct $\Rightarrow\exists\delta\left(\gamma\neq\delta\land D\left(\alpha,\beta;\gamma,\delta\right)\right)$. 
\end{itemize}
\end{defn}
The idea behind this relation is that there is some tree in the background
and for $\alpha,\beta,\gamma,\delta$ points in the tree, $D(\alpha,\beta;\gamma,\delta)$
holds if the shortest path between $\alpha$ and $\beta$ does not
intersect the shortest path between $\gamma$ and $\delta$. See
Figure \ref{fig:D-Relation}.

\begin{figure}[h]
\includegraphics{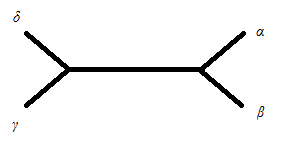} \protect\caption{\label{fig:D-Relation}$D$ relation}
\end{figure}

As in Section \ref{sub:A-cyclic-separation}, we have:
\begin{lem}
Suppose that $D$ is a $D$-relation on a set $X$. Then the following
axiom holds:
\begin{itemize}
\item [\emph{(d6)}]If $D\left(\alpha,\beta;\gamma,\delta\right)$ and $D\left(\alpha,\beta;\gamma,\delta'\right)$
then $D\left(\alpha,\beta;\delta,\delta'\right)$.
\end{itemize}
\end{lem}
\begin{proof}
Use (d3) and (d1) to try and replace $\gamma$ by $\delta'$ in $D\left(\alpha,\beta;\gamma,\delta\right)$
and then to try and replace $\gamma$ by $\delta$ in $D\left(\alpha,\beta;\gamma,\delta'\right)$.
If both fail, then it must be that $D\left(\delta',\beta;\gamma,\delta\right)$
and $D\left(\delta,\beta;\gamma,\delta'\right)$ which together contradict
(d2) and (d1). \end{proof}
\begin{prop}
\label{prop:noD}$\AGL n$ does not preserve a $D$-relation on $\Omega$. \end{prop}
\begin{proof}
Suppose it does. 

Let $L$ be a line in $\Omega$. Let $a,b,c\in L$ be distinct. By
(d5) for some $d\neq c$, $D\left(a,b;c,d\right)$. Note that by (d2)
and (d1), it must be that $d\neq a,b$. If $d\in L$, then by (d3),
we can replace either $d$ or $a$ by some $d'\notin L$. Using (d1)
in the latter case, we may then assume that $D(a,b;c,d)$ holds with
$d\notin L$. Since $\Omega\backslash L$ is a Jordan set, by (d6)
we have that $D\left(a,b;d,d'\right)$ for any $d,d'\notin L$. 

By applying $\AGL n$, it follows that for any $a,b,d$ and $d'$,
if $d$ and $d'$ are not on the line determined by $a,b$ then $D\left(a,b;d,d'\right)$.

So, letting $\set{e_{i}}{i<n}$ be the standard basis of $\Omega$,
let $a=-e_{1}$, $b=-e_{2}$, $d=e_{1}$ and $d'=e_{2}$. But then
let $\sigma\in\GL n$ replace $e_{2}$ by $-e_{2}$ while fixing $e_{1}$.
Then by applying $\sigma$ it follows that $D\left(-e_{1},-e_{2};e_{1},e_{2}\right)$
and $D\left(-e_{1},e_{2};e_{1},-e_{2}\right)$. By (d1) and (d2) this
is a contradiction. 
\end{proof}

\subsection{Steiner systems}
\begin{defn}
\label{def:steiner}Let $k\in\Nn$ be such that $k\geq2$. A \emph{Steiner
$k$-system} $\left(X,\calB\right)$ consists of a set $X$ of \emph{points}
and set $\calB$ of \emph{blocks} where $\calB\subseteq\SS\left(X\right)$,
for all $\bb_{1},\bb_{2}\in\calB$, $\left|\bb_{1}\right|=\left|\bb_{2}\right|>k$
and:
\begin{enumerate}
\item There is more than one block.
\item If $\alpha_{1},\ldots,\alpha_{k}$ are distinct points in $X$, then
there is a unique block $\bb\in\calB$ containing them. 
\end{enumerate}
\end{defn}
\begin{example}
Let $X=\Omega$ and $\calB$ be the set of lines in $X$. Then $\left(X,\calB\right)$
forms a 2-Steiner system. So $\AGL n$ preserves a 2-Steiner system
on $\Omega$.\end{example}
\begin{rem}
If $\left(X,\calB\right)$ is a $k$-Steiner system, then for any
permutation $\sigma$ of $X$, $\sigma$ preserves $\calB$ iff $\sigma$
preserves the $k+1$-ary relation $R$ defined as $R\left(x_{1},\ldots,x_{k+1}\right)$
iff $x_{1},\ldots,x_{k+1}$ lie in the same block.\end{rem}
\begin{lem}
\label{lem:block are contained in lines}If $\left(\Omega,\calB\right)$
is a $k$-Steiner system that $\AGL n$ preserves and $a_{1},\ldots,a_{k}\in\Omega$
are distinct points contained in some affine subspace $A$ of dimension
$<n$, then the block $\bb\in\calB$ they determine is contained in
$A$.\end{lem}
\begin{proof}
Suppose that for some $y\notin A$, $y\in\bb$. Then since $\Omega\backslash A$
is a Jordan set, it follows by (2) in Definition \ref{def:steiner}
that $\bb$ contains $\Omega\backslash A$. Let $L$ be a line disjoint
from $A$ in $\Omega$. So $\bb$ contains $L$ (in particular, $k$
points from $L$), but also some points outside of $L$. By the same
argument, $\bb$ contains $\Omega\backslash L$. Together $\bb=\Omega$
contradicting (1) in Definition \ref{def:steiner}. \end{proof}
\begin{lem}
\label{lem:If G prop. cont. AGL then k>2}If $G$ is any group of
permutations of $\Omega$ properly containing $\AGL n$ which preserves
a $k$-Steiner system on $\Omega$, then $k>2$.\end{lem}
\begin{proof}
By Lemma \ref{lem:block are contained in lines}, if $k=2$, then
blocks are contained in lines. However, by Corollary \ref{cor:3-trans},
the action of $G$ on $\Omega$ is is 3-transitive, so it takes 3
points on the same block (so collinear) to 3 non-collinear points
(so not on the same block). Contradiction.\end{proof}
\begin{lem}
\label{lem:at least k+2}If $\left(\Omega,\calB\right)$ is a $k$-Steiner
system that $\AGL n$ preserves then all blocks $\bb\in\calB$ have
size at least $k+2$.
\end{lem}
Note that by Definition \ref{def:steiner}, blocks must contain at
least $k+1$ points, but this lemma asks for one more. 
\begin{proof}
We divide the proof to two cases:
\begin{casenv}
\item $k=2l+1$ is odd.

Let $\bb\in\calB$ be the unique block containing 
\[
s=\left\{ -le_{1},-\left(l-1\right)e_{1},\ldots,-e_{1},0,e_{1},\ldots,\left(l-1\right)e_{1},le_{1}\right\} ,
\]
where as usual $\set{e_{i}}{i<n}$ is a standard basis for $\Omega$.
Then this block contains at least one more point $x$, which we already
know is on the line $L=\Qq e_{1}$ by Lemma \ref{lem:block are contained in lines}.
Then any map $\sigma\in\GL n$ taking $e_{1}$ to $-e_{1}$ will take
$x$ to $-x$ while fixing $s$, so $-x\in\sigma\left(\bb\right)$
but $\sigma\left(\bb\right)=\bb$ since they both contain $s$. So
$\bb$ contains at least $k+2$ points.

\item $k=2l$ is even.

Let $\bb\in\calB$ be the unique block containing 
\[
s'=\left\{ -le_{1},-\left(l-1\right)e_{1},\ldots,-e_{1},e_{1},\ldots,\left(l-1\right)e_{1},le_{1}\right\} .
\]
Again, $\bb$ contains at least one more point $x$ on this line.
If the argument for the odd case fails, it means that $x=0$. So now
we may assume that $\bb$ contains $s=s'\cup\left\{ 0\right\} $.
Let $\sigma\in\AGL n$ be the map $x\mapsto x+e_{1}$. Then $\left|\sigma\left(s\right)\cap s\right|\geq k$
(namely $\sigma\left(s\right)\cap s=\left\{ -\left(l-1\right)e_{1},\ldots,le_{1}\right\} $),
so $\sigma\left(\bb\right)=\bb$. But then $\bb$ contains $\left(l+1\right)e_{1}$
and we are done. 

\end{casenv}
\end{proof}
\begin{rem}
\label{rem:infinite sequence}Suppose that $L\subseteq\Omega$ is
a line, and that $\sigma\in\AGL n$ with $\sigma\left(L\right)=L$.
If for some $a\neq b,x\in L$, $\sigma\left(a\right)=b$ and $\sigma\left(x\right)=x$,
then unless $x=\left(a+b\right)/2$, $\sequence{\sigma^{m}\left(a\right)}{m<\omega}$
is without repetitions. 

Indeed, after conjugating by an element from $\AGL n$ we may assume
that $L=\Qq e_{1}$, and that $x=0$. Thus it is enough to note that
if $\tau\in\GL 1=\Qq^{\x}$, $r\in\Qq\backslash\left\{ 0\right\} ,$
then for all $m$, $\tau^{m}\left(r\right)=r$ iff $\tau^{m}=\id$
which implies that $\tau=\id$ or $\tau=-\id$.\end{rem}
\begin{prop}
\label{prop:AGL no Steiner}$\AGL n$ does not preserve a $k$-Steiner
system on $\Omega$ for $k>2$.\end{prop}
\begin{proof}
Suppose $\AGL n$ preserves the $k$-Steiner system $\left(\Omega,\calB\right)$
and $k>2$. Let $\bb$ be the block determined by $\set{me_{1}}{0\leq m<k-1}\cup\left\{ e_{2}\right\} $.
By Lemma \ref{lem:at least k+2}, this block contains at least two
more points, $x$ and $y$, and neither of them is in the line $L'=\Qq e_{1}$
(by Lemma \ref{lem:block are contained in lines}). By Lemma \ref{lem:block are contained in lines},
$x$ and $y$ are in the plane spanned by $\left\{ e_{1},e_{2}\right\} $.
Note that in general, for distinct $a,b,c\in\Omega$, if $\left(a+b\right)/2,\left(b+c\right)/2,\left(a+c\right)/2$
are collinear, then $a,b,c$ are collinear and contained in the same
line: this line contains $b$ as $b=\frac{a+b}{2}+\frac{b+c}{2}-\frac{a+c}{2}$,
and similarly $a$ and $c$. It follows that at least for one pair
from $\left\{ e_{2},x,y\right\} $, say $x$ and $y$, $\left(x+y\right)/2\notin L'$. 
\begin{casenv}
\item The line $L$ determined by $y$ and $x$ intersects the line $L'$.

Call the intersection point $x_{0}$. Let $\sigma\in\AGL n$ fix $L'$
and map $x$ to $y$. Then $\sigma$ must fix setwise the line $L$.
By Remark \ref{rem:infinite sequence}, $\sequence{\sigma^{m}\left(x_{0}\right)}{m<\omega}$
is infinite, and contained in $L$. But this contradicts Lemma \ref{lem:block are contained in lines}.

\item $L$ does not intersect $L'$.

Since $x,y$ are in the plane spanned by $\left\{ e_{1},e_{2}\right\} $,
this means that $y-x\in L'$ (i.e., the line determined by $x$ and
$y$ is parallel to $L'$). So if $\sigma\in\AGL n$ fixes $L'$ and
takes $x$ to $y$, then easily $\sequence{\sigma^{m}\left(x\right)}{m<\omega}$
is infinite and contained in the unique line containing $x$ and parallel
to $L'$, and we reach the same contradiction. 

\end{casenv}
\end{proof}
\begin{cor}
\label{cor:no Steiner}If $G$ is a group of permutations of $\Omega$
which properly contains $\AGL n$, then $G$ does not preserve a $k$-Steiner
system for $k\geq2$. \end{cor}
\begin{proof}
Follows from Lemma \ref{lem:If G prop. cont. AGL then k>2} and Proposition
\ref{prop:AGL no Steiner}. 
\end{proof}

\subsection{\label{sub:A-limit-of}A limit of Steiner systems}

The following definition is taken from \cite[Theorem 5.8.4, Theorem 5.8.2]{MR1357089}.
\begin{defn}
Let $\left(G,X\right)$ be a Jordan group. Then $G$ is said to preserve
on $X$ a \emph{limit of Steiner systems} if there is some $3\leq m\in\Nn$,
some totally ordered index set $\left(J,\leq\right)$ with no greatest
element, and a strictly increasing chain $\sequence{\Pi_{i}}{i\in J}$
of subsets of $X$ such that:
\begin{enumerate}
\item $\bigcup\set{\Pi_{i}}{i\in J}=X$;
\item for each $i\in J$, $G_{\left\{ \Pi_{i}\right\} }$ is $\left(m-1\right)$-transitive
on $\Pi_{i}$ (where $G_{\left\{ \Pi_{i}\right\} }$ is the setwise
stabilizer of $\Pi_{i}$), and preserves a non-trivial Steiner $\left(m-1\right)$-system
on $\Pi_{i}$;
\item if $i<j$ then $\Pi_{i}$ is a subset of a block of the $G_{\left\{ \Pi_{j}\right\} }$-invariant
Steiner $\left(m-1\right)$-system on $\Pi_{j}$;
\item for all $g\in G$ there is $i_{0}\in J$, dependent on $g$, such
that for every $i>i_{0}$ there is $j\in J$ such that $g\left(\Pi_{i}\right)=\Pi_{j}$
and the image under $g$ of every $\left(m-1\right)$-Steiner block
on $\Pi_{i}$ is an $\left(m-1\right)$-block on $\Pi_{j}$;
\item for each $i\in J$, the set $X\backslash\Pi_{i}$ is a Jordan set
for $\left(G,X\right)$. 
\end{enumerate}
\end{defn}
We will show that if $G$ is a group of permutations of $\Omega$
properly containing $\AGL n$, and $G$ preserves a limit of Steiner
systems, then if $G$ is not highly transitive, it must already preserve
a $k$-Steiner system for some $k\in\Nn$, contradicting Corollary
\ref{cor:no Steiner}. 

In fact, we will not use the full definition of a limit of Steiner
system. Instead, we will use the following definition:
\begin{defn}
\label{def:weak limit}Let $\left(G,X\right)$ be a Jordan group.
Then $G$ is said to preserve on $X$ a \emph{limit of Jordan sets}
if there is some totally ordered index set $\left(J,\leq\right)$
with no greatest element, and a strictly increasing chain $\sequence{\Pi_{i}}{i\in J}$
of subsets of $X$ such that:
\begin{enumerate}
\item $\bigcup\set{\Pi_{i}}{i\in J}=X$;
\item for all $g\in G$ there is $i_{0}\in J$, dependent on $g$, such
that for every $i>i_{0}$ there is $j\in J$ such that $g\left(\Pi_{i}\right)=\Pi_{j}$;
\item for each $i\in J$, the set $X\backslash\Pi_{i}$ is a Jordan set
for $\left(G,X\right)$. 
\end{enumerate}
\end{defn}
\begin{lem}
\label{lem:closure}Suppose that $G$ preserves a limit of of Jordan
sets on $X$ as witnessed by $\sequence{\Pi_{i}}{i\in J}$. Then for
every $g\in G$, for all $i$ large enough either $g^{-1}\left(\Pi_{i}\right)\subseteq\Pi_{i}$
or $g\left(\Pi_{i}\right)\subseteq\Pi_{i}$.\end{lem}
\begin{proof}
By (2) of Definition \ref{def:weak limit}, for all $i$ large enough,
$g\left(\Pi_{i}\right)=\Pi_{j}$ for some $j\in J$. Fix such $i$,
and suppose that $j\geq i$. Then as $\Pi_{i}\subseteq\Pi_{j}$, $\Pi_{i}\subseteq g\left(\Pi_{i}\right)\Rightarrow g^{-1}\left(\Pi_{i}\right)\subseteq\Pi_{i}$.
If $j<i$, we get that $g\left(\Pi_{i}\right)\subseteq\Pi_{i}$. \end{proof}
\begin{assumption}
\label{asm:weakly limit}Suppose that $G$ is a group of permutation
of $\Omega$, properly containing $\AGL n$ and preserving a limit
of Jordan sets as witnessed by $\left(J,\leq\right)$ and $\sequence{\Pi_{i}}{i\in J}$.
Also, fix some $m\in\Nn$ such that $G$ is $m$-transitive but not
$\left(m+1\right)$-transitive.\end{assumption}
\begin{defn}
Say that a tuple of distinct elements $\bar{a}=\left(a_{0},\ldots,a_{m}\right)\in\left(\Omega\right)^{m+1}$
is \emph{very large} if for some $i\in J$, $a_{0},\ldots,a_{m-1}\in\Pi_{i}$
and $a_{m}\notin\Pi_{i}$. Say that an $\left(m+1\right)$-tuple $\bar{a}$
is \emph{large} if its orbit contains a very large $m+1$-tuple. \end{defn}
\begin{lem}
\label{lem:only one large orbit}The large $\left(m+1\right)$-tuples
consist of one orbit. \end{lem}
\begin{proof}
We need to show that if $\bar{a}$ and $\bar{b}$ are large, then
for some $\sigma\in G$, $\sigma\left(\bar{a}\right)=\bar{b}$. We
may assume that both $\bar{a}$ and $\bar{b}$ are very large. Let
$i\in J$ be such that $a_{0},\ldots,a_{m-1}\in\Pi_{i}$ and $a_{m}\notin\Pi_{i}$.
By $m$-transitivity, for some $\sigma\in G$, $\sigma\left(\bar{a}\upharpoonright m\right)=\bar{b}\upharpoonright m$.
Let $i_{0}\in J$ correspond to (2) of Definition \ref{def:weak limit}
applied to $\sigma$, and let $i'>i_{0},i$. Since $\Omega\backslash\Pi_{i}$
is a Jordan set and $\Pi_{i'}\neq\Omega$, we can fix $\Pi_{i}$ and
move $a_{m}$ out of $\Pi_{i'}$. This allows us to assume that $i>i_{0}$.
So $\sigma\left(\Pi_{i}\right)=\Pi_{j}$ for some $j\in J$, and $\bar{b}\upharpoonright m\subseteq\Pi_{j}$.
By moving $b_{m}$, fixing some $\Pi_{j'}$ containing $\bar{b}\upharpoonright m$,
we may assume that $b_{m}\notin\Pi_{j}$. 

But since $\Omega\backslash\Pi_{j}$ is a Jordan set, and $\sigma\left(a_{m}\right)\notin\Pi_{j}$,
we can map $\sigma\left(a_{m}\right)$ to $b_{m}$ fixing $\Pi_{j}$
via some $\tau\in G$. Then $\tau\circ\sigma\left(\bar{a}\right)=\bar{b}$. \end{proof}
\begin{lem}
\label{lem:last not on line is large}If $\bar{a}=\left(a_{1},\ldots,a_{m+1}\right)\in\Omega^{m+1}$
is such that $\bar{a}\upharpoonright m\subseteq L$ for some line
$L$ and $a_{m+1}$ is not in $L$ then $\bar{a}$ is large.\end{lem}
\begin{proof}
Let $\Pi_{i}$ be such that $\Pi_{i}$ contains $\bar{a}\upharpoonright m$.
We claim that $\Omega\backslash\Pi_{i}\not\subseteq L$. Suppose not,
i.e., that $\Omega\backslash L\subseteq\Pi_{i}$. Let $L'\neq L$
and let $\sigma\in\AGL n$ take $L$ to $L'$. We may assume (perhaps
increasing $i$) that $\sigma\left(\Pi_{i}\right)=\Pi_{j}$ for some
$j\in J$, and so $\Omega\backslash L'\subseteq\Pi_{j}$. Let $j'>i,j$,
then $\Pi_{j'}\supseteq\Omega\backslash A$ where $\left|A\right|\leq1$.
But then $J$ must have a last element. 

Since we can map $a_{m}$ to any point outside of $L$, we can map
it to a point outside of $\Pi_{i}$. \end{proof}
\begin{lem}
\label{lem:permutations}If $\bar{a}\in\Omega^{m+1}$ is large, $\pi\in S_{m+1}$
any permutation, then $\bar{a}^{\pi}=\sequence{a_{\pi\left(i\right)}}{i<m}$
is also large.\end{lem}
\begin{proof}
It is enough to prove it for $\pi$ of the form $\left(k\;\; m\right)$
for some $k<m$. Since large tuples form one orbit (Lemma \ref{lem:only one large orbit}),
it is enough to show the lemma for \uline{one} large tuple. 

We will find a line $L$ and some $i\in J$ such that $L\cap\Pi_{i}$
contains at least $m-1$ points: $a_{0},\ldots a_{k-1},a_{k+1},\ldots,a_{m-1}$,
and both $L\backslash\Pi_{i}$ and $\Pi_{i}\backslash L$ are nonempty,
say containing $x$ and $a_{k}$ respectively. Then $\bar{a}=\left(a_{0},\ldots,a_{m-1},x\right)$
is very large by definition, but also $\bar{a}^{\pi}$ is large by
Lemma \ref{lem:last not on line is large}. Together we are done. 

First assume that $n=\omega$, and let $\set{e_{j}}{j\in\Zz}$ be
the standard basis for $\Omega$ enumerated by $\Zz$. Let $\sigma_{0}\in\GL n$
take $e_{j}$ to $e_{-j}$ (so $\sigma^{2}=\id$). By Lemma \ref{lem:closure},
for all $i$ large enough, $\Pi_{i}$ contains $0$ and it is closed
under $-\id$ and under $\sigma_{0}$. 

By applying Lemma \ref{lem:closure} with $\sigma$ being translation
by $e_{1}$, $\sigma\left(x\right)=x+e_{1}$, for all large enough
$i$, $\Pi_{i}$ is closed under translation by $e_{1}$ \uline{and}
by $-e_{1}$. Indeed, if for instance $\Pi_{i}$ is closed under $\sigma$,
then $\left(-\id\right)\sigma\left(-\id\right)\left(\Pi_{i}\right)\subseteq\Pi_{i}$
but $\left(-\id\right)\sigma\left(-\id\right)=\sigma^{-1}$. 

By applying Lemma \ref{lem:closure} with $\sigma$ being base shift,
i.e., $\sigma\in\GL n$ and $\sigma\left(e_{j}\right)=e_{j+1}$, for
all $i$ large enough $\Pi_{i}$ is closed under base shift \uline{and}
its inverse. This follows again from the fact that $\sigma^{-1}=\sigma_{0}\sigma\sigma_{0}$. 

Now it follows that for $i$ large enough, $\Pi_{i}$ is closed under
translation by $\pm e_{j}$ for all $j\in\mathbb{Z}$. Indeed, suppose
$\sigma_{j}$ is translation by $e_{j}$, and $\tau_{j\to1}\in\GL n$
takes $e_{k}$ to $e_{k-j+1}$. Then $\sigma_{j}=\tau_{j\to1}^{-1}\sigma_{1}\tau_{j\to1}$. 

But now, for all $i$ large enough, $\Pi_{i}$ contains all the integer
valued linear combinations of $\set{e_{j}}{j\in\Zz}$, which is just
$\Zz^{<\omega}$. 

Let $x\in\Omega\backslash\Pi_{i}$. Then, if $p\in\Nn$ is the product
of the denominators of the rationals appearing in $x$, then the line
$L=\Zz px$ intersects $\Pi_{i}$ in infinitely many points. In particular,
we can find $m-1$ points on $L\cap\Pi_{i}$, and find some $a_{k}\in\Zz^{<\omega}\backslash L$.

The proof for finite $n$ is similar but simpler, since we do not
need to use $\sigma_{0}$, as the base shift modulo $n$ map $\sigma\in\GL n$
is of finite order hence $\sigma\left(\Pi_{i}\right)=\Pi_{i}$ for
all large enough $i\in J$, so automatically $\Pi_{i}$ is closed
under $\sigma^{-1}$. 
\end{proof}
In light of Lemma \ref{lem:permutations} we can extend the definition
of large tuples to $\left(m+1\right)$-sets. We will define an $m$-Steiner
system on $\Omega$. For a subset of $s\subseteq\Omega$ of size $m$,
let 
\[
\bb_{s}=\set{x\in\Omega}{s\cup\left\{ x\right\} \mbox{ is not large}}.
\]
Equivalently, $\bb_{s}$ is the set of points $x\in\Omega$ such that
whenever $s$ is sent by $G$ to some line, $x$ is sent to the same
line (this follows from Lemma \ref{lem:only one large orbit} and
Lemma \ref{lem:last not on line is large}). It follows that $s\subseteq\bb_{s}$. 
\begin{lem}
\label{lem:exchange}Assume $s\subseteq\Omega$ is of size $m$, and
let $a\in s$. Then if $b\in\bb_{s}$ then $\bb_{s}=\bb_{\left(s\backslash\left\{ a\right\} \right)\cup\left\{ b\right\} }$. \end{lem}
\begin{proof}
Let $t=\left(s\backslash\left\{ a\right\} \right)\cup\left\{ b\right\} $.
To see that $\bb_{t}\subseteq\bb_{s}$, take some $x\in\bb_{t}$,
and some $\sigma\in G$ that sends $s$ to some line $L$. Then $\sigma\left(b\right)\in L$,
so $\sigma\left(t\right)\subseteq L$, so $\sigma\left(x\right)\in L$.

For the other direction, assume that $x\in\bb_{s}$ and that $\sigma\in G$
maps $t$ to a line $L$. Since $t\cup\left\{ a\right\} =s\cup\left\{ b\right\} $
is not large (since $b\in\bb_{s}$), it follows that $\sigma\left(a\right)\in L$,
and thus $\sigma\left(x\right)\in L$. 
\end{proof}
Let 
\[
\calB=\set{\bb_{s}}{s\subseteq\Omega,\left|s\right|=m}.
\]

\begin{cor}
\label{cor:noLimit}$\left(\Omega,\calB\right)$ is an $m$-Steiner
system on $\Omega$ preserved by $G$.\end{cor}
\begin{proof}
Note that by definition, $G$ preserves $\calB$ and $\sigma\left(\bb_{s}\right)=\bb_{\sigma\left(s\right)}$
for all $\sigma\in G$ and $s\subseteq\Omega$ of size $m$. By $m$-transitivity,
it follows that $\left|\bb_{1}\right|=\left|\bb_{2}\right|$ for any
$\bb_{1},\bb_{2}\in\calB$. We already noted that $\left|\bb\right|\geq m$
for all $\bb\in\calB$. If $\left|\bb\right|=m$ for all $\bb\in\calB$,
then all $\left(m+1\right)$-sets are large, and by Lemma \ref{lem:only one large orbit},
this would mean that $G$ is $\left(m+1\right)$-transitive. This
shows that $\left|\bb\right|>m$ for all $\bb\in\calB$. 

Since there is a large tuple, there is more than one block.

Finally we must check that if $s\subseteq\Omega$ is of size $m$,
and $\bb\in\calB$ contains $s$ then $\bb=\bb_{s}$. Suppose $\bb=\bb_{t}$
for some $t$. Then $s\subseteq\bb_{t}$, so by Lemma \ref{lem:exchange},
we can replace every element of $t$ by an element of $s$ until we
get that $\bb_{s}=\bb_{t}$.
\end{proof}

\subsection{Conclusion}
\begin{cor}
Theorem \ref{thm:main} holds: $\AGL n$ is a maximal-closed subgroup
of $S_{\omega}$ for $\omega\geq n\geq2$. \end{cor}
\begin{proof}
Suppose that $G$ is some group of permutations of $\Omega$ strictly
containing $\AGL n$. 

By Corollary \ref{cor:3-trans}, $G$ is 3-transitive. Since $G$
is a Jordan group as witnessed by e.g., lines, we may apply Fact \ref{fac:Macpharson-and-Adeleke}.
By Proposition \ref{prop:No S relation}, Proposition \ref{prop:noD}
and Corollary \ref{cor:no Steiner}, $G$ must preserve a limit of
Steiner systems. But by Corollary \ref{cor:noLimit}, which assumes
even the weaker hypothesis of Assumption \ref{asm:weakly limit},
unless $G$ is highly transitive, $G$ preserves an $m$-Steiner system
where $G$ is $m$-transitive but not $\left(m+1\right)$-transitive
(so $m\geq3$) and this contradicts Corollary \ref{cor:no Steiner}.
\end{proof}
Recall from the introduction the question of Junker and Ziegler \cite{DBLP:journals/jsyml/JunkerZ08}
: does every non-$\omega$-categorical theory admit infinitely many
reducts? Bodirsky and Macpherson \cite{Bodirsky2013} have answered
this question negatively. Our next corollary (which was noticed by
David Evans) gives yet another counterexample. This also answers a
question from \cite{Bodirsky2013}.
\begin{defn}
\label{def:reduct}A structure $N$ is a \emph{(definable) reduct}
of a structure $M$ if they share the same universe, and the basic
relations of $N$ are $\emptyset$-definable subsets of $M$. The
structure $N$ is a \emph{proper reduct} if there are $\emptyset$-definable
subsets of $M$ which are not $\emptyset$-definable in $N$. For
a complete first order theory $T$ in the language $L$, we say that
a (complete first order) theory $T'$ in the language $L'$ is a (proper)
reduct of $T$ if there is a model $M$ of $T$ and a model $N$ of
$T'$ such that $N$ is a (proper) reduct of $M$. Equivalently, every
model of $T$ has a (proper) reduct that is a model of $T'$. The
theory $T'$ is a \emph{trivial} reduct when $T'$ is the theory of
an infinite set with no structure. \end{defn}
\begin{cor}
\label{cor:David Evans}The theory $T=Th\left(\Qq,f\right)$ where
$f\left(x,y,z\right)=x+y-z$ has no non-trivial proper definable reduct. \end{cor}
\begin{proof}
Let $L=\left\{ f\right\} $. Then $T$ is a reduct of $T_{0}$, the
theory of a divisible torsion free abelian group (or a vector space
over $\Qq$) in the language $\left\{ +,0\right\} $. The vector space
$\Qq^{<\omega}$ is a countable saturated model of $T_{0}$ (since
$T_{0}$ is strongly minimal, or by quantifier elimination). Thus
its reduct $M$ to $L$ (so $M\models T$) is also saturated. Suppose
that $M$ has a proper reduct $N$ with language $L'$. Since $N$
is a proper reduct, by saturation there are two tuples $\bar{a},\bar{b}$
in $N$ such that $\bar{a}\equiv_{L'}\bar{b}$ but $\bar{a}\not\equiv_{L}\bar{b}$.
By saturation of $N$, there is an automorphism of $N$ that maps
$\bar{a}$ to $\bar{b}$, thus the automorphism group of $N$ is strictly
larger than $\Aut\left(M\right)=\AGL{\omega}$. By Theorem \ref{thm:main},
$\Aut\left(N\right)$ is the full permutation group of $\Qq^{<\omega}$,
thus $N$ is the trivial structure. 
\end{proof}

\section{\label{sec:PGL}$\protect\PGL n$ is maximal}

In this section we will show using the same techniques as in Section
\ref{sec:AGL} that $\PGL n$ is maximal for all $n\leq\omega$. Recall:
\begin{defn}
For a vector space $V$ over a field $F$, let $P\left(V\right)$
be the set of one-dimensional subspaces of $V$. Let $PGL\left(V\right)$
be the group $GL\left(V\right)/Z\left(GL\left(V\right)\right)$ (where
$Z\left(GL\left(V\right)\right)$ is just the group $\set{\alpha\id}{\alpha\in F^{\times}}$).
Elements of $P\left(V\right)$ are called \emph{points}, while two-dimensional
subspaces of $V$ are called \emph{lines}. A point $x$ is lies on
a line $L$ if $x\subseteq L$. For $n<\omega$, let $\PGL n=PGL\left(\Qq^{n}\right)$,
and let $\PGL{\omega}=PGL\left(\Qq^{<\omega}\right)$. 
\end{defn}
With this definition of lines, points and incidence, $P\left(V\right)$
satisfies all the axioms of a projective space \cite{MR1629468},
and $PGL\left(V\right)$ is a group of automorphism of the projective
space structure. In fact, by Fact \ref{fac:fund-projective} below,
this is the group of all automorphisms of the projective space in
case the field is $\Qq$ (or in general when $\Aut\left(F\right)$
is trivial). 

As in Section \ref{sec:AGL}, we assume:
\begin{assumption}
\label{asm:main projective assumption}Throughout this section, we
fix $3\leq n\leq\omega$. Let $V=\Qq^{n}$ as a vector space in case
$n<\omega$ and $\Qq^{<\omega}$ in case $n=\omega$. Let $\Omega=P\left(V\right)$.\end{assumption}
\begin{fact}
\label{fac:fund-projective}\cite[Corollary 3.5.9]{MR1629468} (The
fundamental theorem of projective geometry) For a permutation $\sigma$
of $\Omega$, $\sigma\in\PGL n$ iff $\sigma$ preserves lines iff
$\sigma$ takes triples of collinear points to collinear points. 
\end{fact}
It follows from Fact \ref{fac:fund-projective} that $\PGL n$ is
a closed subgroup of the group of permutations of $\Omega$ (which
is countable) --- it is the automorphism group of the structure $\left(\Omega,R\right)$
where $R\left(x,y,z\right)$ holds iff $x,y,z$ are collinear. 
\begin{rem}
\label{rem:Affine subspaces}Suppose that $U$ is a hyperplane in
$V$ (a subspace of co-dimension $1$). So $U$ can be seen as an
affine space with its structure of lines and points. 

Let $G=G_{U}=\GL n_{\left\{ U\right\} }$ --- the setwise stabilizer
of $U$. Let $PG=G/Z\left(G\right)$ (where $Z\left(G\right)=\set{\alpha\id}{\alpha\in\Qq^{\times}}$).
Let $X=X_{U}=\set{x\in\Omega}{x\subseteq U}$. Then $PG$ acts on
$\Omega\backslash X$. This action is equivalent to the action of
$\AGL{n-1}$ on $U$ (if $n=\omega$, $n-1=\omega$). To see this,
choose a basis $B$ of $U$ and $b\in V$ such that $B\cup\left\{ b\right\} $
is a basis of $V$. Present each point $x$ in $\Omega\backslash X$
uniquely as $\left\langle b+u_{x}\right\rangle $ --- the span of
$b+u_{x}$ --- where $u_{x}\in U$. Similarly, present each $\sigma\in PG$
uniquely as $T\cdot Z\left(G\right)$ where $T\in G$ with $T\left(b\right)=b+u_{\sigma}$
for $u_{\sigma}\in U$. Then $\sigma\left(x\right)=\left\langle T\left(b+u_{x}\right)\right\rangle =\left\langle Tb+Tu_{x}\right\rangle =\left\langle b+u_{\sigma}+Tu_{x}\right\rangle $.
Thus, by identifying $x$ with $u_{x}$ and $\sigma$ with the map
$u\mapsto u_{\sigma}+Tu$ we get the desired equivalence. 

Moreover, the identification of $\Omega\backslash X$ with the affine
space $U$ (via $x\mapsto u_{x}$) preserves lines: collinear points
in $\Omega$ map to collinear points in $U$. In fact, it preserves
affine/projective subspaces as well, meaning that if $W\subseteq V$
is a subspace then $P\left(W\right)\cap\left(\Omega\backslash X\right)$
is an affine subspace of $\Omega\backslash X$ (or $\emptyset$),
and for any affine subspace $A\subseteq\Omega\backslash X$, the projective
subspace $P\left(W\right)$ generated by $A$ in $\Omega$ intersects
$\Omega\backslash X$ in $A$. 

Easily, $\Omega$ is covered by affine spaces, and if $B$ is a basis
of $V$ then $\Omega=\bigcup_{b\in B}\left(\Omega\backslash X_{U_{b}}\right)$
where $U_{b}$ is the hyperplane spanned by $B\backslash\left\{ b\right\} $.
After choosing $U$, we call $\Omega\backslash X$ the corresponding\emph{
affine space} and $X$ its \emph{hyperplane at infinity}. 

If $\Omega\backslash X$ is an affine space, and $L$ is a line in
it, then (the prolongation of) $L$ intersects $X$ in exactly one
point, which we call \emph{the point at infinity} of $L$. It follows
that if $\sigma$ is an affine map of $\Omega\backslash X$ that preserves
$L\cap\left(\Omega\backslash X\right)$ then the unique extension
of $\sigma$ to $\PGL n$ preserves the point at infinity of $L$.

Note that since $V$ cannot be covered by finitely many hyperplanes,
it follows that for any finite set $s\subseteq V\backslash\left\{ 0\right\} $
there is some hyperplane $U$ such that $s\cap U=\emptyset$. This
means that given finitely many points from $\Omega$, there is some
hyperplane $U$ such that all these points are in $\Omega\backslash X_{U}$. \end{rem}
\begin{thm}
\label{thm:main proj}Under Assumption \ref{asm:main projective assumption},
$\PGL n$ is a maximal-closed permutation group of $\Omega$. \end{thm}
\begin{proof}
The proof follows the same lines as the proof of Theorem \ref{thm:main}.
We will go over the steps, using the notation of Remark \ref{rem:Affine subspaces}
and the notation of the proof of Theorem \ref{thm:main}. 

Step 1: Analyze the orbits of $\PGL n$ on triples of distinct elements
from $\Omega$. Here, as opposed to the situation in Lemma \ref{lem:orbits},
there are only two orbits --- the set of collinear triples and the
set of non collinear triples. Hence Corollary \ref{cor:3-trans} follows
at once from Fact \ref{fac:fund-projective}: if $G$ is a group of
permutations of $\Omega$ properly containing $\PGL n$ then it is
3-transitive. 

Step 2: Observe that $\PGL n$ is a Jordan group. Indeed the complement
of a line or of any proper projective subspace is a Jordan set. Now
we may apply Fact \ref{fac:Macpharson-and-Adeleke}.

Step 3: We deal with the $S$ and $D$-relations exactly as we did
in Sections \ref{sub:A-cyclic-separation} and \ref{sub:A-D-relation},
by working within an affine space. In the $S$-relation case, we can
first fix any hyperplane $U$, and choose our line and points in $\Omega\backslash X_{U}$.
Since the action of $PG_{U}\leq\PGL n$ on $\Omega\backslash X_{U}$
is equivalent to the action of $\AGL n$ on $U$, we get that $PG_{U}$
cannot preserve an $S$-relation on $\Omega\backslash X_{U}$, so
the same follows for $\PGL n$ and $\Omega$ (because the axioms for
the $S$-relation are universal). In the $D$-relation case, we first
choose a line $L$ and we get that $D\left(a,b;c,d\right)$ holds
for $a,b,c\in L$ and $d\notin L$. Then we can choose a hyperplane
$U$ such that $a,b,c,d\in\Omega\backslash X_{U}$, and we work within
$\Omega\backslash X_{U}$ with $PG_{U}$ and reach a contradiction
as in Section \ref{sub:A-D-relation}. 

Step 4: We deal with the Steiner system case.

Lemma \ref{lem:block are contained in lines} remains true by replacing
affine subspace by projective subspace $A$ with similar proof: instead
of choosing a line $L$ disjoint of $A$, choose a line $L$ such
that $\left|A\cap L\right|\leq1$. But if $\bb$ is a block which
contains $\Omega\backslash\left\{ x\right\} $ for some $x$, then
by transitivity, $\bb=\Omega$. Lemma \ref{lem:If G prop. cont. AGL then k>2}
follows. 

For Lemma \ref{lem:at least k+2}, we have to reverse the even and
odd cases. Start by working within some affine space $\Omega\backslash X_{U}$
as in Step 3, which we identify with $\Qq^{n-1}$ (or $\Qq^{<\omega}$).
For even $k=2l\geq4$, choose $k-1$ points in $\Omega\backslash X_{U}$,
$-\left(l-1\right)e_{1},\ldots,0,\ldots,\left(l-1\right)e_{1}$ and
then add the point at infinity of the line containing these points.
Now there is one more point, $x$, which must be on the line $L$
and hence must be in $\Omega\backslash X_{U}$. The affine map $-\id$
then preserves $L$, and hence also the point at infinity, but takes
$x$ to $-x$, hence adds one more point. For odd $k=2l+1\geq3$,
choose $k-1$ points in $\Omega\backslash X_{U}$, $-le_{1},\ldots,-e_{1},e_{1},\ldots,le_{1}$
and the point at infinity of this line. Now there is one more point
$x$ which again must be in the line but also in $\Omega\backslash X_{U}$,
so either this point is $\neq0$, in which case we can proceed as
in the odd case, or this point is $0$, in which case we can translate
by $e_{1}$, fixing the point at infinity, and get one more point. 

Proposition \ref{prop:AGL no Steiner} now follows with the same proof.
First we choose any $k-1$ points on some line, and we add one more
point out of it. We know that there are two more points which belong
to the projective space (in fact a plane) generated by the $k$ chosen
points. Choose some hyperplane $U$ such that $\Omega\backslash X_{U}$
contains all these points. So now we work within the affine space
$\Omega\backslash X_{U}$, and produce infinitely many points in a
block, all contained in the same line. Note that in the proof there,
the choice of points on the line was not important. 

This shows that any group $G\gneq\PGL n$ cannot preserve a $k$-Steiner
system on $\Omega$ for $k\geq2$. 

Step 5: The limit of Steiner system case.

The proof as in Section \ref{sub:A-limit-of} works with some minor
modifications. So again we show that if $\PGL n$ preserves a limit
of Jordan sets on $\Omega$ and it is $m$-transitive but not $\left(m+1\right)$-transitive
then it must preserve an $ $$m$-Steiner system on $\Omega$. 

Lemmas \ref{lem:closure} and \ref{lem:only one large orbit} hold
with exactly the same proofs. Also Lemma \ref{lem:last not on line is large}
(which only used the fact the two lines intersect in at most one point). 

Lemma \ref{lem:permutations} requires some small modification. The
proof uses exactly the same technique, but now we first choose a basis
$B$ of $V$, and we show that for some $i\in J$ large enough, for
all $b\in B$, $\Pi_{i}$ contains all the points with integer coefficients
in the affine space which corresponds to $b$ (i.e., which corresponds
to the hyperplane spanned by $B\backslash\left\{ b\right\} $) under
the identification described in Remark \ref{rem:Affine subspaces}.
These points correspond to points which admit a tuple of homogeneous
coordinates consisting of integers, at least one of which is 1. Since
any point of $\Omega$ belongs to at least one of these affine spaces,
the same proof will work.

Suppose first that $n=\omega$, and let $\set{e_{i}}{i\in\mathbb{Z}}$
be a basis for $V$. Denote by $A_{i}$ the affine space corresponding
to $e_{i}$. Using the same technique of the proof of Lemma \ref{lem:permutations},
we find $i\in J$ large enough so that $\Pi_{i}$ is preserved under
the projective maps induced by the linear maps $\sigma_{1}$, which
maps $e_{i}\mapsto e_{i+1}$, its inverse, $\sigma_{2}$ which fixes
$e_{0}$, maps $e_{i}\mapsto e_{i+1}$ for $i\neq-1$, and maps $e_{-1}\mapsto e_{1}$,
and its inverse. We also want $\Pi_{i}$ to be preserved under the
affine map of translating by $e_{1}$ in the affine space that corresponds
to $e_{0}$, and its inverse. It follows that with the map $\sigma_{1}$
we cover all affine spaces, and with the map $\sigma_{2}$ we cover
all the coordinates within one affine space. Now if $i$ is large
enough so that $\Pi_{i}$ contains the $0$ point of the affine space
corresponding to $e_{0}$, it will follow that $\Pi_{i}$ is as required. 

For $n$ finite the proof is the same (but simpler). 

The rest of the proof is exactly the same, so the theorem follows. \end{proof}
\begin{rem}
\label{rem:general field}The same proof goes through to show that
$P\Gamma L_{n}\left(K\right)$ is maximal for any field $K$ of characteristic
$0$ in which the only roots of unity are $\left\{ 1,-1\right\} $
(though Lemma \ref{lem:permutations} needs a slightly different argument).
However, recall from the introduction that Bogomolov and Rovinsky
\cite{bogomolovRovinsky} proved that for $n\geq3$, $\PGaL nF$ is
a maximal-closed subgroup of the group of all permutations of $F$
for any infinite field $F$. 
\end{rem}

\section{Open questions}
\begin{question}
Is $\PGaL 2F$ maximal for an algebraically closed field $F$ of transcendence
degree $\geq1$?
\end{question}
Note that the action of $\PGaL 2F$ on $P\left(F^{2}\right)$ preserves
the $3$-Steiner system whose blocks are conjugates of $\acl\left(\Qq\right)\cup\left\{ \infty\right\} $.

\begin{question}
Is the automorphism group of the geometry of a homogeneous structure
constructed using the Hrushovski construction maximal-closed?
\end{question}
This question might be a bit too vague, so here is a precise special
case. Consider the Hrushovski construction giving a homogeneous 2-Steiner
system described in \cite[Chapter 15]{MR1632579} or \cite[Section 5]{MR1226304}.
Namely, let $R$ be a ternary relation symbol, and define a pre-dimension
$\delta$ on the class of 3-hypergraphs ($R$-structures in which
$R$ is symmetric and holds only for tuples of distinct points) by
$\delta\left(A\right)=\left|A\right|-\left|R^{A}\right|$. Consider
the family of 3-hypergraphs $A$ for which $\delta\left(A_{0}\right)\geq\min\left\{ \left|A_{0}\right|,2\right\} $
for any $A_{0}\subseteq A$. The usual Fra\"iss\'e-Hrushovski amalgamation
associated with this class and pre-dimension gives a countable structure
$M$ equipped with a dimension function $d$. Consider the permutations
of $M$ which preserve the dimension. Is this group maximal?

After this paper appeared we were told in private communication by
Omer Mermelstein that this group is not maximal-closed. According
to him, a similar construction to the one done in \cite[Section 6]{omer}
gives a counterexample. 
\begin{question}
Is there a $k$-Steiner system on a countable set whose automorphism
group $G$ is $k$-transitive and also preserves an $l$-steiner system
for $l\neq k$. One may also add the condition that the $k$-blocks
are Jordan complements for $G$. 

\bibliographystyle{alpha}
\bibliography{common}
\end{question}

\end{document}